\documentclass{article}

\usepackage{amsfonts}
\usepackage{amsthm}
\usepackage{amsmath}
\usepackage{xcolor}
\usepackage{comment}
\usepackage{enumerate}

\newtheorem{Thm}{Theorem}[section]
\newtheorem{Lem}[Thm]{Lemma}
\newtheorem{Def}[Thm]{Definition}
\newtheorem{Prop}[Thm]{Proposition}
\newtheorem{Cor}[Thm]{Corollary}
\newtheorem{ThmL}{Theorem}

\newtheorem*{Rem}{Remark}

\newcommand{\mb}{\mathbb}

\newcommand{\ol}{\overline}

\def\sideremark#1{\ifvmode\leavevmode\fi\vadjust{\vbox to0pt{\vss
 \hbox to 0pt{\hskip\hsize\hskip1em
 \vbox{\hsize3cm\tiny\raggedright\pretolerance10000
 \noindent #1\hfill}\hss}\vbox to8pt{\vfil}\vss}}}

\numberwithin{equation}{section}

\title{The Yamabe flow on asymptotically Euclidean manifolds with nonpositive Yamabe constant}
\date{\vspace{-5ex}}
\author{Gilles Carron\thanks{Nantes Universit\'e; gilles.carron@univ-nantes.fr; research partially supported by ANR grants ANR-18-CE40-0012:RAGE
 and ANR-17-CE40-0034:CCEM }, Eric Chen\thanks{University of California, Berkeley, ecc@berkeley.edu; research partially supported by NSF Award DMS-3103392},  Yi Wang\thanks{Johns Hopkins University, ywang261@jhu.edu;
research partially supported by NSF CAREER Award DMS-1845033}}

\begin{document}
\maketitle

\begin{abstract}
We study the Yamabe flow on asymptotically flat manifolds with non-positive Yamabe constant $Y\leq 0$. Previous work by the second and third named authors \cite{ChenWang} showed that while the Yamabe flow always converges in a global weighted sense when $Y>0$, the flow must diverge when $Y\leq 0$. We show here in the $Y\leq 0$ case however that after suitable rescalings, the Yamabe flow starting from any asymptotically flat manifold must converge to the unique positive function which solves the Yamabe problem on a compactification of the original manifold.
\end{abstract}

\section{Introduction}

In this article we continue the study of the convergence of the Yamabe flow
\begin{align}
\begin{cases}
\frac{\partial g}{\partial t}=-R g,\label{floweq}
\\
g(0)=g_0,
\end{cases}
\end{align}
starting from an asymptotically flat (AF) manifold $(M^n,g_0)$. Above, $R$ denotes the scalar curvature of the Riemannian metric $g=g(t)$. This flow preserves the conformal class of $g_0$ in the sense that $g(t)\in[g_0]$ for all times $t$, and is the natural analogue of the volume-normalized Yamabe flow on compact manifolds introduced by Hamilton \cite{Hamilton}.
It is well known that on a compact manifold, the normalized Yamabe flow is the gradient flow of the Einstein--Hilbert functional within a fixed conformal class. It can be viewed as a natural evolution equation which could potentially evolve a given metric to a constant scalar curvature metric within the same conformal class.
For long-time existence and convergence of the Yamabe flow on compact manifolds, we refer interested readers to the work of Hamilton, Chow, Ye, Schwetlick--Struwe, and Brendle \cite{Hamilton, Chow2, Ye,Struwe,Brendle5,Brendle6}.
The study on noncompact manifolds is less developed. On noncompact manifolds, long-time existence has been proved under some assumptions of suitable pointwise bounds on curvature and conformal factors. See for instance \cite{MaLi}, \cite{Schulz}.
Other works give long-time existence results in the settings of conformally hyperbolic and singular spaces \cite{Schulz2,Bahuaud,Olsen}. Similar to the compact case, convergence results for the Yamabe flow on noncompact manifolds have been slower to develop---we are aware of \cite{MaLi,MaLi21} in which $C^\infty_{loc}$ convergence to a scalar flat limit metric is shown, using crucially an assumption that the initial metric has non-negative scalar curvature, as well as \cite{CLV}, which studies convergence of Yamabe flow on singular spaces with positive Yamabe constant.

The study of Yamabe flow on asymptotically flat manifolds was initiated by \cite{CZ}. They proved short-time existence and that asymptotic flatness is preserved under the flow. They also discussed the ADM mass under the flow.

In a previous article \cite{ChenWang} by the second and third authors, we proved the long-time existence of Yamabe flow \eqref{floweq} on asymptotically flat manifolds. Moreover, we showed that the flow converges in a global weighted sense (defined by \cite{Bartnik86}) if and only if the Yamabe constant $Y(M^n,[g_0])$ is positive. Long-time existence was also studied independently by \cite{MaLi21}, who also considered local convergence assuming nonnegative scalar curvature.
The convergence/divergence behavior of the Yamabe flow on 
asymptotically flat metrics is quite different from that of the Ricci flow; see for example \cite{YuLi} regarding the Ricci flow in this setting in dimension $n=3$. We also refer readers to related results on the Ricci flow in \cite{DaiMa,EC}. 


For convenience, we recall the main theorems in \cite{ChenWang} here, referring to Section \ref{subsec_def} for definitions and notation related to asymptotically flat manifolds. When $Y(M^n,[g_0])> 0$, the flow converges in a weighted global sense. 
\begin{Thm}[{\cite[Theorem 1.3]{ChenWang}}]\label{Ypos}
Let $(M^n,g_0)$, $n\geq 3$ be a $C^{k+\alpha}_{-\tau}$ AF manifold with $Y(M,[g_0])>0$, $k\geq 3$, and  $\tau>1$. Then there exists a Yamabe flow $(M^n,g(t))$ starting from $(M^n,g_0)$ defined for all positive times and a metric $g_\infty$ on $M^n$ which is $C^{k+\alpha}_{-\tau'}$ AF for all $\tau'<\min\{\tau,n-2\}$ so that for any such $\tau'$ we have
\begin{align}
\|g(t) - g_\infty \|_{ C^{k+\alpha}_{-\tau'} }=O(t^{- \delta_0  }),  \quad  \mbox{as} \ t\rightarrow \infty,
\end{align}
for some $\delta_0>0$. In particular, this Yamabe flow converges in $C^{k+\alpha}_{-\tau'}$ to the asymptotically flat, scalar flat metric $g_\infty$.
\end{Thm}

This theorem holds for $n\geq 3$, but in order to make sense the ADM mass of an asymptotically flat manifold when $n=3$, we have a different version in the $n=3$ case by adding the natural conditions $R_{g_0}\geq 0$ and $R_{g_0}\in L^1$ (if we are concerned with the mass).



\begin{Thm}[{\cite[Theorem 1.4]{ChenWang}}]\label{Ypos3}
Let $(M^3,g_0)$ be a $C^{k+\alpha}_{-\tau}$ AF manifold with $Y(M,[g_0])>0$, $k\geq 3$, $\tau>\frac{1}{2}$, $R_{g_0}\geq 0$, and $R_{g_0}\in L^1$. Then there exists a Yamabe flow $(M^n,g(t))$ starting from $(M^n,g_0)$ defined for all positive times and a metric $g_\infty$ on $M^n$ which is $C^{k+\alpha}_{-\tau'}$ AF for all $\tau'<\min\{\tau,1\}$ so that for any such $\tau'$ we have
\begin{align}
\|g(t) - g_\infty \|_{ C^{k+\alpha}_{-\tau'} }=O(t^{- \delta_0  }),  \quad  \mbox{as} \ t\rightarrow \infty,
\end{align}
for some $\delta_0>0$. In particular, this Yamabe flow converges in $C^{k+\alpha}_{-\tau'}$ to the asymptotically flat, scalar flat metric $g_\infty$.
\end{Thm}

In contrast to the positive Yamabe case, when $Y(M, [g_0])\leq 0$, while the flow still exists for all positive times, it must diverge.

\begin{Thm}[{\cite[Theorem 1.2 (2)]{ChenWang}}]\label{uconverge}
Let $(M^n,g_0)$ be a $C^{k+\alpha}_{-\tau}$ AF manifold with $k\geq 3$.
If $Y(M^n,[g_0])\leq 0$, then the Yamabe flow $(M^n,g(t))$ starting from $(M^n,g_0)$ does not converge. In particular, $g(t)=u(t)^{\frac{4}{n-2}}g_0$ will fail to remain uniformly equivalent to $g_0$ as $t\rightarrow\infty$, and both  $\|u(t)\|_{L^\infty}$ and the $L^2$ Euclidean-type Sobolev constant of $g(t)$ will tend to positive infinity.
\end{Thm}

Here $Y(M,[g_0])$ is a conformally invariant quantity. Motivated by the definition of the Yamabe constant in the compact case, $Y(M,[g_0])$ is defined as follows:
\begin{align}
Y(M,[g_0]):=\inf_{\substack{v\in C^{\infty}_0(M),\\{v\neq 0}}} \frac{\int_M a_n |\nabla v|_{g_0}^2+R_{g_0} v^2\ dV_{g_0}}{\left(\int |v|^{\frac{2n}{n-2}}\ dV_{g_0}\right)^{\frac{n-2}{n}}},
\end{align}
with $a_n=\frac{4(n-1)}{n-2}$.
This Yamabe constant plays an important role in the prescribed scalar curvature problem on conformal classes of asymptotically flat metrics \cite{CantorBrill,Maxwell,DiltsMaxwell}.

The main goal of this article is to  study the behavior of the flow in this latter case $Y(M^n,[g_0])\leq 0$. More precisely, we prove that even though the Yamabe flow diverges (which in our case is equivalent to saying the solution $u$ blows up as $t\rightarrow \infty$, 
by the proof of \cite[Lemma 3.4]{ChenWang}), the rescaled flow 
$t^{-\frac{n-2}{4}}u(x, t)$ is convergent when $Y(M^n,[g_0])\leq 0$. Working from here, we obtain precise profiles of the blow up behavior of the flow.

Recall that if we write a Yamabe flow as $g(x,t)=u(x,t)^{\frac{4}{n-2}}g_0$, then $u$ satisfies the parabolic equation
\begin{equation}
\frac{\partial}{\partial t} u^{\frac{n+2}{n-2}}=\frac{n+2}{4}(a_n\Delta_{g_0} u-R_{g_0}u).\label{flow_equation}
\end{equation}
Sometimes we will also write $N=\frac{n+2}{n-2}$ for the exponent on the left. To simplify writing the right-hand side, we denote by $L_{g_0}$ the conformal Laplacian
\[L_{g_0}:=-a_n\Delta_{g_0}+R_{g_0}.\]
Throughout what follows, $u$ will always denote a solution of this equation with $\lim_{|x|\rightarrow\infty}u(x,t)=1$.
By \cite[Theorem 1.3]{CZ}, $g(t)$ remains asymptotically flat along the Yamabe flow if $g_0$ is asymptotically flat. 

There are two main results of this paper---one when  $Y(M^n,[g_0])< 0$, and
the other when $Y(M^n,[g_0])= 0$. In both these cases,  the $u(x, t)$ associated with the Yamabe flow $(M^n,g(t))$ blows up at a rate no faster than $ O(t^{\frac{n-2}{4}})$. But the limiting profiles of $\tilde{u}(x, t):=    t^{-\frac{n-2}{4}}u(x, t)$ behave differently. 

\begin{ThmL}\label{rescalednega}
If $(M^n,g_0)$ is a $C^{k+\alpha}_{-\tau}$ AF manifold with $k\geq 3$ and $Y(M^n,[g_0])< 0$,
then the Yamabe flow $(M^n,g(t))$ starting from $(M^n,g_0)$ blows up at the rate $u(x, t)= O(t^{\frac{n-2}{4}})$.
Moreover, $\tilde{u}(x, t):=    t^{-\frac{n-2}{4}}u(x, t)$ converges in $C^{k,\alpha'}_{loc}$ for any $\alpha'<\alpha$ to a limiting function $\tilde{u}_\infty>0$, where $\tilde{u}_\infty$ is the unique solution to the equation of prescribed constant scalar curvature $-1$:
\begin{align}
\begin{cases}
-a_n\Delta_{g_0} \tilde{u}_\infty+R_{g_0} \tilde{u}_\infty=-\tilde{u}_\infty^N, \label{eq_steady_sec1}\\
\tilde{u}_\infty \rightarrow 0.
\end{cases}
\end{align}
Moreover, $\tilde{u}_\infty(x)$ satisfies the sharp spatial decay $\tilde{u}_\infty(x)= O(|x|^{2-n})$, and $\tilde{u}_\infty^{\frac{4}{n-2}}g_0$ extends to the unique 
$-1$ constant scalar curvature metric on the compactified space $\ol{M}$.

\end{ThmL}

When $Y(M^n,[g_0])= 0$, the rescaled function $\tilde{u}(x,t)$ once again has a limit.
But in this case the limiting function vanishes on $M^n$.
\begin{ThmL}\label{rescaledzero}
If $(M^n,g_0)$ is a $C^{k+\alpha}_{-\tau}$ AF manifold with $k\geq 3$ and $Y(M^n,[g_0])= 0$,
then the Yamabe flow $(M^n,g(t))$ starting from $(M^n,g_0)$ satisfies $u(x, t)= o(t^{\frac{n-2}{4}})$. 
\end{ThmL}
Thus in order to describe the blow-up profile of the limit, we need a more delicate estimate. We prove in the following that the limit of the flow $u(x,t)$, after being renormalized by its maximum value on a compact set $K$ at $t$ is convergent. Moreover, such a limit is the positive canonical solution (up to a multiplicative constant) of a prescribed zero scalar curvature equation. 
\begin{ThmL}\label{rescaledzerolimit}
If $(M^n,g_0)$ is a $C^{k+\alpha}_{-\tau}$ AF manifold with $k\geq 3$ and $Y(M^n,[g_0])= 0$,
then for the Yamabe flow $(M^n,g(t))$ starting from $(M^n,g_0)$ and any fixed compact set $K\subset M^n$ we have that $\frac{u(x,t)}{\max_{x\in K}u(x,t)}$ converges in $C^{k,\alpha'}_{loc}$ for any $\alpha'<\alpha$ to the unique positive solution $w(x)$ on $(M^n,g_0)$ which satisfies
\begin{align}
-a_n \Delta_{g_0} w + R_{g_0} w =0,\quad \max_{x\in K}w(x)=1.
\end{align}
Moreover, $w(x)$ satisfies the sharp spatial decay $w(x)= O(|x|^{2-n})$, and $w^{\frac{4}{n-2}}g_0$ extends to the constant zero scalar curvature metric on the compactified space $\ol{M}$, which is unique up to scaling.
\end{ThmL}

\begin{Rem}
It is worth noting that it is unclear whether $\max_{x\in M}u(x,t)$ may always be attained on some fixed compact set. However, given any compact set $K$, we can construct auxiliary functions $v_b(x, t)$ and $v_B(x, t)$ to bound $u(x, t)$ from above and below. Moreover $v_b$ and $v_B$ both take their maximum values in $M$ on $K$. These maximum values are suitable for use in the renormalization, and $\max_{x\in M}u(x,t)$ will be no larger than a fixed constant multiple of $\max_{x\in K}u(x,t)$.
\end{Rem}

\subsection{Organization of the article}
The organization of the article is as follows:
In Section \ref{sec_prelim} we start by recalling some preliminaries for the Yamabe flow on AF manifolds as well as properties of certain compactifications of manifolds with $Y(M^n,[g_0])\leq 0$. In Section \ref{sec_Y<0} we discuss the rescaled convergence of Yamabe flows starting from AF manifolds with $Y(M^n,[g_0])<0$ and prove Theorem \ref{rescalednega}. In Section \ref{sec_Y=0} we discuss why the same rescaling does not give a nontrivial convergence result when $Y(M^n,[g_0])=0$, proving Theorem \ref{rescaledzero}, and then describe another rescaling which does yield convergence to a smooth positive function, and prove Theorem \ref{rescaledzerolimit}.


\section{Preliminaries}\label{sec_prelim}

After recalling some relevant definitions and notation, we describe in this section some properties which hold for any Yamabe flow starting from an asymptotically flat manifolds. In the last part we discuss the existence of Yamabe metrics on certain compactifications of asymptotically flat manifolds with $Y(M^n,[g_0])\leq 0$.

\subsection{Asymptotically flat manifolds}\label{subsec_def}

Here we recall as in \cite{ChenWang} some standard function spaces and related definitions used in the analysis and definition of asymptotically flat (AF) manifolds.  See for instance \cite{Bartnik86,DiltsMaxwell}.

\begin{Def}\label{funcspaces}
Let $M^n$ be a complete differentiable manifold such that there exists a compact $K\subset M^n$ and a diffeomorphism $\Phi:M^n\backslash K\rightarrow \mb{R}^n\backslash B_{R_0}(0)$, for some $R_0>0$. Let $r\geq 1$ be a smooth function on $M^n$ that agrees under the identification $\Phi$ with the Euclidean radial coordinate $|x|$ in a neighborhood of infinity, and let $\hat{g}$ be a smooth metric on $M^n$ which is equal to the Euclidean metric in a neighborhood of infinity under the identification $\Phi$. Then with all quantities below computed with respect to the metric $\hat{g}$, we have the following function spaces:

The weighted Lebesgue spaces $L^q_\beta(M)$, for $q\geq 1$ and weight $\beta\in\mb{R}$, consist of those locally integrable functions on $M$ such that the following respective norms are finite:
\begin{align}\notag
\|v\|_{L_{\beta}^{q}(M)}=\left\{\begin{array}{ll}{\left(\int_{M}|v|^{q} r^{-\beta q-n} d x\right)^{\frac{1}{q}},} & {q<\infty}, \\ {\operatorname{ess} \sup _{M}\left(r^{-\beta}|v|\right),} & {q=\infty}.\end{array}\right. 
\end{align}

The weighted Sobolev spaces $W^{k,q}_\beta(M)$ are then defined in the usual way with the norms
\begin{align}\notag
\|v\|_{W_{\beta}^{k, q}(M)}=\sum_{j=0}^{k}\left\|D_{x}^{j} v\right\|_{L_{\beta-j}^{q}(M)}.
\end{align}

The weighted $C^k$ spaces $C^k_\beta(M)$ consist of the $C^k$ functions for which the following respective norms are finite:
\begin{align}\notag
\|v\|_{C_{\beta}^{k}(M)}=\sum_{j=0}^{k} \sup _{M} r^{-\beta+j}\left|D_{x}^{j} v\right|.
\end{align}

The weighted H\"{o}lder spaces $C_\beta^{k+\alpha}(M)$, $\alpha\in(0,1)$, consist of those $v\in C_\beta^k(M)$ for which the following respective norms are finite:
\begin{align}\notag
\|v\|_{C_{\beta}^{k+\alpha}(M)}=\|v\|_{C_{\beta}^{k}(M)}+\sup _{x \neq y \in M} \min (r(x), r(y))^{-\beta+k+\alpha} \frac{\left|D_{x}^{k} v(x)-D_{x}^{k} v(y)\right|}{d(x,y)^{\alpha}}.
\end{align}
\end{Def}


An asymptotically flat manifold is then a smooth manifold with an asymptotically flat metric.

\begin{Def}[Asymptotically flat metrics]
Given $M^n$ as in Definition \ref{funcspaces}, a metric $g$ is said to be a $W^{k,q}_{-\tau}$ (respectively $C^k_{-\tau}$, $C^{k+\alpha}_{-\tau}$) asymptotically flat (AF) metric if $\tau>0$ and
\begin{align}
g-\hat{g}\in W^{k,q}_{-\tau}(M)\quad\text{(respectively $C^k_{-\tau}(M)$, $C^{k+\alpha}_{-\tau}(M)$)}.
\end{align}
The number $\tau>0$ is called the order of the asymptotically flat metric.
\end{Def}

\subsection{Comparison principle}

We wish to be able to compare solutions of the parabolic equation \eqref{flow_equation}. The proof of the result below follows some of the arguments in \cite[Lemma 1.4]{Schulz}.

\begin{Lem}\label{max_principle_lem}
Let $u_1,u_2$ be two positive solutions of \eqref{flow_equation} for all $t\geq 0$. Suppose moreover that
\begin{enumerate}[(1)]
\item For all $T>0$, there exists $C_T>0$ such that 
\[0<C_T^{-1}\leq u_1,u_2\leq C_T\quad\text{on }M\times[0,T].\]
\item\label{max_cond2} For all $t\geq0$,
\[\lim_{|x|\rightarrow\infty}u_1(x,t)<\lim_{|x|\rightarrow\infty}u_2(x,t),\]
and these limits are achieved uniformly in space on $[0,T]$.
\item\label{max_cond3} We have $u_1(x,0)\leq u_2(x,0)$.
\end{enumerate}
Then for all $t\geq 0$ we have $u_1(x,t)\leq u_2(x,t)$.
\end{Lem}
\begin{proof}

Now by conditions \eqref{max_cond2} and \eqref{max_cond3}, for any $(x,T)$ there exists a set $\Omega\subset M$ such that $x\in \Omega$ and $u_1\leq u_2$ on $(\Omega\times\{0\})\cup(\partial\Omega\times[0,T])$. Indeed, those conditions imply that $\ell_i(t)=\lim_{x\rightarrow\infty}u_i(x,t)$ for $i=1,2$ are continuous functions, so that $\inf_{t\in[0,T]}\lim_{x\rightarrow\infty}u_2(x,t)-u_1(x,t)>0$. Since these limits are achieved uniformly in space, this last inequality allows us to find $\Omega$ sufficiently large with the desired properties.


Then by the linear parabolic maximum principle \cite[\S 3.3]{Protter} we must have $u_1\leq u_2$ in $\Omega\times[0,T]$. Since $(x,T)$ was arbitrary the result follows.
\end{proof}

\subsection{Growth control and rescaled solutions}\label{subsec_growth_control}

As mentioned in the Introduction, from \cite{ChenWang} we know that the solutions $u(x,t)$ of \eqref{flow_equation} corresponding to Yamabe flows starting from AF manifolds with $Y\leq 0$ must blow up. Below we first observe that a standard estimate on the evolution of the scalar curvature allows us to control the growth of $u$ in general.

\begin{Lem}\label{lem_R_lower}
Let $(M^n,g_0)$ be a $C^{2+\alpha}_{-\tau}$ AF manifold. Along the Yamabe flow $(M^n,g(t))$, we have
\[R(g(t))\geq-\frac{1}{t}.\]
\end{Lem}
\begin{proof}
Recall that under the Yamabe flow we have $\frac{\partial}{\partial t} R=(n-1)\Delta_{g_t} R+R^2$, and let $\alpha<0$ be such that $R_{g_\rho}\geq\alpha$ when $t=0$. Choose an $\epsilon>0$, and set
\[\phi(t)=\frac{\alpha}{\epsilon-\alpha t},\]
which satisfies
\[\frac{\partial}{\partial t}\phi=\phi^2=(n-1)\Delta_{g_t}\phi+\phi^2.\]
Then
\[\frac{\partial}{\partial t}(\phi-R)=(n-1)\Delta_{g_t}(\phi-R)+(\phi+R)(\phi-R),\]
and we may apply the Ecker--Huisken maximum principle \cite[Theorem 4.3]{EckerHuisken} on $M\times[0,T]$ for any $T>0$ to conclude that $\phi-R\leq 0$ on $M\times[0,\infty)$. Taking $\epsilon\rightarrow 0$ yields $R(x,t)\geq-\frac{1}{t}$.
\end{proof}

\begin{Prop}\label{upper_bd_lem}
Let $u(x,t)$ be a solution of \eqref{flow_equation} corresponding to the Yamabe flow starting from a $C^{2+\alpha}_{-\tau}$ AF manifold. Then $t^{-\frac{n-2}{4}} u(x,t)$ is nonincreasing, and thus has a nonnegative limit, so that
\[\max_{x\in M} u(x,t)=O(t^{\frac{n-2}{4}}).\]
\end{Prop}
\begin{proof}
Integrate the equation
\[\frac{\partial}{\partial t} u=-\frac{n-2}{4} R u\]
and apply the estimate of Lemma \ref{lem_R_lower}.
\end{proof}


In light of Proposition \ref{upper_bd_lem}, we define
\begin{equation}\label{eq_utilde}
\tilde{u}(x,t)=t^{-\frac{n-2}{4}}u(x,t).
\end{equation}
Since $\tilde{u}(x,t)\geq 0$ and is monotonically decreasing, it has a pointwise limit as $t\rightarrow\infty$, and we set
\begin{equation}
\tilde{u}_\infty(x)=\lim_{t\rightarrow\infty}\tilde{u}(x,t).
\end{equation}
When $Y(M^n,[g_0])<0$ we will see that $\tilde{u}_\infty>0$ is a smooth function solving a naturally associated elliptic equation, while when $Y(M^n,[g_0])=0$ we will find that $\tilde{u}_\infty\equiv 0$ converges to zero. In this latter case, a different normalization of $u$ yields a (smooth) positive limit.

\subsection{Conformally compactifying $M$}\label{subsec_compactify}

In order to carry out our study of the cases $Y(M^n,[g_0])\leq 0$, we check below in this setting that there exists a conformal compactification of $M$ which topologically is given by adding a point at infinity to $M$, and which solves the Yamabe problem on this compactified manifold, so that it has constant scalar curvature $Y(M^n,[g_0])$. 


First we have the following result on compactifying $M$ by adding a single point at infinity, without yet adding a requirement on its scalar curvature. See also \cite{Herzlich} for related compactification results.

\begin{Lem}[{\cite[Lemmas 5.2, 5.3]{DiltsMaxwell}}]\label{lem_DiltsMaxwell}
Suppose $(M,g)$ is a $W^{2,p}_{-\tau}$ AE manifold, with $p>n/2$ and $\tau\geq\frac{n}{p}-2$. Then there exists a smooth conformal factor $\phi$ decaying to zero at the rate $r^{2-n}$ such that $\ol{g}=\phi^{\frac{4}{n-2}}g$ extends to a $W^{2,p}$ metric on the one-point compactification $\ol{M}=M\cup\{q\}$. Moreover, $Y(M,g)=Y(\ol{M},[\ol{g}])$.
\end{Lem}
\begin{Rem}
Note that $\ol{g}$ is smooth away from $q$; the metric $\ol{g}$ is $W^{2,p}$ in the sense that in a coordinate ball $B$ about $q$, the components of $\ol{g}$ along with its first and second derivatives are $L^p$ integrable with respect to $dV_{\ol{g}}$. Hence the metric $\ol{g}$ is H\"{o}lder continuous and bilipschitz equivalent to a smooth metric on the manifold $\ol{M}$---that is to say there is a smooth Riemannian metric $\underline{g}$  on $\ol{M}$ such that for some positive constant $\gamma:$
$$\gamma^{-2} \underline{g}\le \ol{g}\le \gamma^2  \underline{g}.$$
Hence
$$ \gamma^{-n} dV_{\underline{g}}\le dV_{\ol{g}}\le \gamma^{n} dV_{\underline{g}}\quad\text{and}\quad\gamma^{-1}d_{\underline{g}}\le d_{\ol{g}}\le \gamma d_{\underline{g}}.$$ Moreover for any smooth function $\varphi:$

$$\gamma^{-1}| d\varphi|_{\underline{g}}\le | d\varphi|_{\ol{g}}\le \gamma| d\varphi|_{\underline{g}}.$$
So the spaces $L^p(\ol{M},  \underline{g})$ and $L^p(\ol{M},  \ol{g})$ are the same with equivalent norms, and the same is true for the $W^{1,2}$-spaces.
\end{Rem}

Next, we need to solve the Yamabe problem on $(\ol{M},[\ol{g}])$. We apply a result of \cite{ACM} to do so.

\begin{Lem}\label{lem_Yamabe_solution}
Suppose $(M,g)$ is a $W^{2,p}_{-\tau}$ AE manifold for some $p>n/2$ and $\tau\geq\frac{n}{p}-2$, with $Y(M,[g])<0$. Then in the notation of Lemma \ref{lem_DiltsMaxwell}, on $(\ol{M},\ol{g})$ there exists a function $u\in W^{1,2}\cap L^\infty(\ol{M})$ with $\inf_{\ol{M}}u>0$ such that on $\ol{M}\setminus\{q\}$,
\begin{equation}
    -a_n\ol{\Delta} u+\ol{R} u=Y(M,[g])u^{\frac{n+2}{n-2}}.\label{eq_Yamabe_solution}
\end{equation}
Here $\ol{\Delta}$ and $\ol{R}$ denote the Laplacian operator and the scalar curvature of $\ol{g}$ respectively.
\end{Lem}
\begin{proof}
The result follows from \cite[Theorem 1.12, Proposition 1.15]{ACM} once we verify its assumptions are satisfied in our case. In the notation of that work we check:
\begin{enumerate}[i)]
\item\label{lem_Yamabe_solution_list} Let $f:\ol{M}\rightarrow\mb{R}$ be a Lipschitz function with respect to the distance induced on $\ol{M}$ metric $\ol{g}$. We need to check that $f$ can be approximated in $W^{1,2}(\ol{M},\ol{g})$ by functions in $C_0^1(\ol{M}\setminus\{q\})$.

Let $\phi:[0,\infty)\rightarrow\mb{R}$ be a smooth, nonnegative function with $\phi(x)=0$ for $x\leq 1$ and $\phi(x)=1$ for $x\geq 2$. With the notation $r_p=d_{\ol{g}}(p,q)$, we define
\[f_\epsilon(p)=f(p)\phi\left(\frac{r_p}{\epsilon}\right).\]

This is true because $\ol{g}$ is a continuous metric, hence bilipschitz to a smooth metric on $\ol{M}$. The desired property certainly holds on smooth metrics, and both function spaces are invariant under bilipschitz equivalence.

\item Since $(\ol{M},\ol{g})$ is compact with finite volume, it suffices to check for $\mu$ the measure induced by $\ol{g}$ that
\begin{equation}
C^{-1} r^n\leq\mu(B(p,r))\leq C r^n\label{eq_Ahlfors}
\end{equation}
for $r>0$ small. Since $\ol{g}$ is bilipschitz to a smooth metric on $\ol{M}$ as noted above, this follows immediately.

\item We check that the Sobolev inequality holds on $W^{1,2}(\ol{M},\ol{g})$. Once again this holds because $\ol{g}$ is bilipschitz to a smooth metric on $\ol{M}$.

\item\label{ACM_4} a) Since $\ol{g}$ is a $W^{2,p}$ metric, we indeed have $\ol{R}\in L^p(\ol{M},d\mathrm{vol}_{\ol{g}})$, where $p>n/2$.
\end{enumerate}
Finally, since $Y(\ol{M},[\ol{g}])=Y(M,g)<0$, then by \cite[Section 1.2]{ACM} the condition $Y(\ol{M},[\ol{g}])<Y_\ell(\ol{M},[\ol{g}])$ from \cite[Theorem 1.12]{ACM} trivially holds, so we can apply that result to obtain the existence of the desired $u$, with \cite[Proposition 1.15]{ACM} giving its positivity.
\end{proof}

\begin{Rem}
Note that the function $u$ given by Lemma \ref{lem_Yamabe_solution} above is smooth away from $q$. Indeed, $\ol{R}$ and the coefficients of $\ol{\Delta}$ are smooth away from $q$, so this follows by the elliptic regularity and boundedness of $u$.
\end{Rem}

When $Y(M,[g])=0$, the analogue of Lemma \ref{lem_Yamabe_solution}  also holds.

\begin{Lem}\label{lem_Yamabe_solution0}
Suppose $(M,g)$ is a $W^{2,p}_{-\tau}$ AE manifold for some $p>n/2$ and $\tau\geq\frac{n}{p}-2$, with $Y(M,[g])=0$. Then in the notation of Lemma \ref{lem_DiltsMaxwell}, on $(\ol{M},\ol{g})$ there exists a function $u\in W^{1,2}\cap L^\infty(\ol{M})$ with $\inf_{\ol{M}}u>0$ such that on $\ol{M}\setminus\{q\}$,
\begin{equation}
-a_n\ol{\Delta} u+\ol{R} u=Y(M,[g])u^{\frac{n+2}{n-2}}=0.\label{eq_Yamabe_solution0}
\end{equation}
\end{Lem}
\begin{proof}
For any open set $\mathcal U\subset M$, we define as in \cite{ACM} Sobolev constant 
$$ S(\mathcal U)=\inf \{\int |d\phi|^2 d\mu: \phi \in W^{1,2}_0(\mathcal U\cap \Omega), \|\phi\|_{\frac{2n}{n-2}}=1\}$$
and local Sobolev constant
$$S_\ell(M,g)=\inf_{p\in M} \lim_{r\rightarrow 0}
S(B(p, r)).
$$
We claim that in the notation of Lemmas \ref{lem_DiltsMaxwell} and \ref{lem_Yamabe_solution} and \cite{ACM}, we have 
$S_\ell(\ol{M},\ol{g})>0$. In a similar way as in the proof of Lemma \ref{lem_Yamabe_solution}, this is true because $\ol{g}$ is bilipschitz equivalent to a smooth metric on $\ol{M}$ and because the desired property holds for smooth metrics.

Consequently, we can follow the arguments of Lemmas \ref{lem_DiltsMaxwell} and \ref{lem_Yamabe_solution}
to obtain a function $u_0$ such that the metric $\hat{g}=u_0^{\frac{4}{n-2}}g_0$ has constant scalar curvature $0$, replacing in the arguments of Lemma \ref{lem_Yamabe_solution} condition \ref{ACM_4}) a) by condition iv) c) of \cite[Theorem 1.12]{ACM}, the negative part $\ol{R}^-\in L^p(\ol{M},d\mathrm{vol}_{\ol{g}})$, and using our claim which gives $Y(\ol{M},[\ol{g}])=0<S_\ell(\ol{M},[\ol{g}])$. 
\end{proof}

We also have uniqueness of the functions given in Lemmas \ref{lem_Yamabe_solution} and \ref{lem_Yamabe_solution0}.

\begin{Prop}\label{prop_Yamabe_unique}
The function $u$ given by Lemma \ref{lem_Yamabe_solution} is unique among functions in $W^{1,2}\cap L^\infty(\ol{M})$, while the function $u$ given by Lemma \ref{lem_Yamabe_solution0} is unique up to a constant multiplicative factor among functions in $W^{1,2}\cap L^\infty(\ol{M})$.
\end{Prop}
\begin{proof}

Let $u_1, u_2\in W^{1,2}\cap L^\infty(\ol{M})$ be two weak solutions of \eqref{eq_Yamabe_solution} or \eqref{eq_Yamabe_solution0}, respectively. Then $v:=\frac{u_1}{u_2}$ satisfies
\begin{equation}
-a_n\Delta_{g_2}v=Y(v^{\frac{n+2}{n-2}}-v),\label{eq_vunique}
\end{equation}
where $\Delta_{g_2}$ is the Laplacian for the conformal metric $g_2=u_2^{\frac{4}{n-2}} \ol{g}$; note that $g_2$ is also a $W^{2,p}$ metric.
The usual arguments in the smooth case to establish uniqueness can be adapted in the non-smooth setting.
In fact the equation \eqref{eq_vunique} holds weakly, meaning that
\begin{equation}\label{weak}\int_{\ol{M}} a_n\langle d \varphi, dv\rangle_{g_2}dV_{g_2}=Y\int_{\ol{M}}\varphi(v^{\frac{n+2}{n-2}}-v) dV_{g_2}\end{equation}
for any $\varphi\in W^{1,2}(\ol{M})$.
We can then test against
$\varphi=\max\{v,1\}$ and obtain that
\begin{equation}\label{ineq}\int_{\left\{ v\ge 1\right\} } a_n|d v|^2_{g_2}dV_{g_2}=Y\int_{\left\{ v\ge 1\right\}}v \,(v^{\frac{n+2}{n-2}}-v) dV_{g_2}+Y \int_{\left\{ v\le 1\right\}} \,(v^{\frac{n+2}{n-2}}-v) dV_{g_2}.\end{equation}
This identity \eqref{ineq} follows from the truncation properties that holds in a fairly general setting (see for instance \cite[Subsection 4.1]{Strum}). In our case it can be justified as follows: let $$\varphi_\epsilon=\frac{1}{2}\left( v-1+\sqrt{(v-1)^2+\epsilon^2}\right).$$ The chain rule implies that $\varphi_\epsilon$ is also in $W^{1,2}$ and
$$d\varphi_\epsilon =\frac12 \left(1+\frac{v-1}{\sqrt{(v-1)^2+\epsilon^2}}\right) dv=\frac{\varphi_\epsilon}{\sqrt{(v-1)^2+\epsilon^2}} dv.$$
Then testing \eqref{weak} using $\varphi_\epsilon$ and letting $\epsilon\to 0+$ implies formula \eqref{ineq}. 
Next if $Y<0$, testing \eqref{weak} with $\varphi=1$ implies
that
\begin{equation}\label{int1}
\int_{\ol{M}}\left(v^{\frac{n+2}{n-2}} -v\right)dV_{g_2}=0.\end{equation}
Hence  $\int_{\left\{ v\ge 1\right\}} \,(v^{\frac{n+2}{n-2}}-v) dV_{g_2} =-\int_{\left\{ v\le 1\right\}} \,(v^{\frac{n+2}{n-2}}-v) dV_{g_2}$ and the equality \eqref{ineq} gives that 
$$\int_{\left\{ v\ge 1\right\} } a_n|d v|^2_{g_2}dV_{g_2}=Y\int_{\left\{ v\ge 1\right\}}(v-1) \,(v^{\frac{n+2}{n-2}}-v) dV_{g_2}.$$
Thus one gets that  $d\varphi=0$ and $v\le 1$, and so when $Y<0$, formula \eqref{int1} implies then that  $v=1$. When $Y=0$, we can directly test \eqref{weak} with $v$ to see that $v$ must be constant.

\end{proof}

By Lemma \ref{lem_DiltsMaxwell} along with either Lemma \ref{lem_Yamabe_solution} or Lemma \ref{lem_Yamabe_solution0}, we may now take $u_0$ to be the smooth function on $M=\ol{M}\setminus\{q\}$ given by
\begin{equation}
u_0=\phi u.\label{eq_u0}
\end{equation}
Here, $\phi$ is defined as in Lemma \ref{lem_DiltsMaxwell}, and $u$ is given by either Lemma \ref{lem_Yamabe_solution} or \ref{lem_Yamabe_solution0}, and satisfies
\[-a_n\Delta_{g}u_0+R_{g}u_0=Y(M^n,[g])u_0^{\frac{n+2}{n-2}}.\]
Note that in the asymptotically Euclidean $z^i$ coordinates on $M$ we have the sharp spatial decay estimate
\[0<C^{-1}<\frac{u_0}{\frac{1}{|z|^{n-2}}}<C\]
for some $C>0$, as $|z|\rightarrow\infty$.

We denote by $\hat{g}$ the metric on $\ol{M}$ given by $\hat{g}=u_0^{\frac{4}{n-2}} g$. Later, when we compare $u_0$ and $\tilde{u}_\infty$, we will need the existence of the Green's function associated with the Dirichlet problem for the Laplacian of $(\ol{M},\hat{g})$ on an open neighborhood about $q$. This follows from a standard result which we quote below.

\begin{Prop}[{\cite[Chapter 10, Sections 9--10]{Sauvigny}}]\label{prop_Green}
Let $g$ be a Riemannian metric which is $C^{1,\alpha}$ for some $\alpha\in(0,1)$ on the closure of a coordinate neighborhood $\Omega$ with smooth boundary.  Then for all $p\in\Omega$ there exists a function
\[G_p:\ol{\Omega}\rightarrow\mb{R}\]
such that $G_p\in W_0^{1,p}(\Omega)\cap C^{2+\alpha}(\Omega\setminus\{p\})$ for $p\in[1,\frac{n}{n-1})$ satisfying the growth conditions
\[0<G_p(x)\leq c_2 d_{\hat{g}}(p,x)^{2-n},\quad x\in\Omega,\quad x\neq p\]
and
\[G_p(x)\geq c_1 d_{\hat{g}}(p,x)^{2-n},\quad x\in\Omega,\quad d_{\hat{g}}(p,x)\leq\frac{1}{2} d_{\hat{g}}(p,\partial\Omega).\]
Moreover, in the sense of distributions $G_p$ satisfies 
\[-\Delta G_p=\delta_p.\]
\end{Prop}

\section{The case $Y<0$}\label{sec_Y<0}

In the case $Y(M^n,[g_0])<0$, we compare the conformal factor $u_0$ as defined in \eqref{eq_u0} to the rescaled Yamabe flow solution $\tilde{u}$ as defined in \eqref{eq_utilde}, and find that $\tilde{u}$ in fact converges to a scalar multiple of $u_0$, which will prove Theorem \ref{rescalednega}. For convenience we will often write $Y$ instead of $Y(M^n,[g_0])$. We first show that $|Y|^{\frac{n-2}{4}}u_0$ always bounds $\tilde{u}$ from below.


\begin{Lem}\label{lower_bd_lem}
If $Y<0$, then $|Y|^{\frac{n-2}{4}}u_0(x)\leq \tilde{u}(x,t) $ for all times $t$.
\end{Lem}
\begin{proof}
Let $u_1(x,t)=(t|Y|)^{\frac{n-2}{4}}u_0(x)$, and define $g_{u_1}=u_1^{\frac{4}{n-2}}g_0$. Then we have $R_{g_{u_1}}=(t|Y|)^{-1}R_{g_{u_0}}=-t^{-1}$, and
\[\frac{\partial}{\partial t} u_1=\frac{n-2}{4}t^{-1}(t|Y|)^{\frac{n-2}{4}}u_0(x)=-\frac{n-2}{4}R_{g_{u_1}} (t|Y|)^{\frac{n-2}{4}}u_0=-\frac{n-2}{4}R_{g_{u_1}}u_1.\]
Equivalently, $u_1$ is a solution of the Yamabe flow. Therefore by the asymptotics of $u_0,u_1$, we can apply the comparison principle of Lemma \ref{max_principle_lem} to conclude that $u_1(x,t)\leq u(x,t)$ on $M\times[0,\infty)$ (we know $u_1\leq u$ for small times, and $u_1\leq u$ as $|x|\rightarrow \infty$ for each fixed $t>0$, so we can apply the comparison starting at a sufficiently small positive time). Rewriting, we see that $|Y|^{\frac{n-2}{4}}u_0(x)\leq t^{-\frac{n-2}{4}}u(x,t)=\tilde{u}(x,t)$.
\end{proof}

By Proposition \ref{upper_bd_lem} and Lemma \ref{lower_bd_lem} together, we see that the pointwise limit $\tilde{u}_\infty(x)$ of $\tilde{u}(x,t)$ as $t\rightarrow\infty$ satisfies $\tilde{u}_\infty(x)\geq |Y|^{\frac{n-2}{4}}u_0(x)$.
We now continue our study of the convergence of $\tilde{u}(x, t)$ and properties of $\tilde{u}_\infty(x)$. 

\begin{Prop}\label{prop_utildeconverge}
We have that $\tilde{u}(x,t)$ converges to $\tilde{u}_\infty(x)$ in $C^{k,\alpha}_{loc}$ as $t\rightarrow\infty$, and $\tilde{u}_\infty(x)$ satisfies
\begin{equation}
-a_n\Delta_{g_0} \tilde{u}_\infty+R_{g_0} \tilde{u}_\infty=-\tilde{u}_\infty^N.\label{eq_steady}
\end{equation}
\end{Prop}
\begin{proof}
Transform the time parameter $t$ to $s$ via the relation $t=e^s$, for $s\in(-\infty,\infty)$. Then
\begin{align*}
\frac{\partial}{\partial s}\tilde{u}^N&=e^s\frac{\partial}{\partial t}\left( t^{-\frac{n+2}{4}} u^N\right)
\\
&=t\left(-\frac{n+2}{4}t^{-\frac{n+2}{4}-1} u^N+ t^{-\frac{n+2}{4}}\frac{n+2}{4}\left(a_n\Delta_{g_0} u-R_{g_0} u\right)\right)
\\
&=\frac{n+2}{4}\left(a_n\Delta_{g_0} \tilde{u}-R_{g_0} \tilde{u}-\tilde{u}^N\right).
\end{align*}
Since $\tilde{u}$ is monotonically decreasing and bounded from below by $|Y|^{\frac{n-2}{4}}u_0$, we have that $\tilde{u}$ is bounded from above and below uniformly in time for $t>0$ on any compact region $\Omega\subset M$. Therefore from the above computation we see that $\tilde{u}(x,e^s)$ satisfies a uniformly parabolic equation on any such $\Omega$, $s\in (s_0, +\infty)$ for $s_0\in\mb{R}$, and therefore by the Krylov--Safonov and (higher order) Schauder estimates for parabolic equations
\[\|\tilde{u}(x,e^s)\|_{C^{k,\alpha}(\Omega\times[s,s+1])}\leq C(\Omega).\]
Consequently by Arzela-Ascoli we obtain for some sequence $\{t_j\}$ with $t_j\rightarrow\infty$ that
\[\tilde{u}(x,t_j)\rightarrow \tilde{u}_\infty(x)\]
in $C^{k,\alpha'}_{loc}$, for any $\alpha'<\alpha$. We can deduce more generally since $\tilde{u}$ is monotonically decreasing that $\tilde{u}(x,t)$ converges to $\tilde{u}_\infty$ in $C^{k,\alpha'}_{loc}$ as $t\rightarrow\infty$. As a result, we see that $\tilde{u}_\infty$ satisfies the steady state equation \eqref{eq_steady}, which is also satisfied by $u_0$. 
\end{proof}

We now check that $\tilde{u}_\infty$ also decays at spatial infinity.

\begin{Lem}\label{lem_compactify}
We have $\tilde{u}_\infty(x)\rightarrow 0$ as $|x|\rightarrow\infty$.
\end{Lem}
\begin{proof}
Suppose the property does not hold. Then there exists a sequence of points $\{x_i\}$ with $|x_i|\rightarrow\infty$ and some $\epsilon>0$ such that $\tilde{u}_\infty(x_i)\geq\epsilon>0$. Since $\tilde{u}$ decreases monotonically to $\tilde{u}_\infty$, this implies
\[\epsilon\leq\tilde{u}_\infty(x_i)\leq t^{-\frac{n-2}{4}}u(x_i,t),\quad\text{for all }x_i\text{ and }t>0.\]
For $t$ sufficiently large, we therefore have
\[2\leq t^{\frac{n-2}{4}}\epsilon\leq u(x_i,t),\]
which contradicts the fact that $u(x,t)\xrightarrow{|x|\rightarrow\infty}1$ for all $t>0$.
\end{proof}


Because $\tilde{u}_\infty\xrightarrow{|z|\rightarrow\infty}0$, we will have on $(\ol{M},\hat{g})$ that
\begin{equation}
\left(\frac{\tilde{u}_\infty}{|Y|^{\frac{n-2}{4}}u_0}\right)(p)=o(d(p,q)^{2-n})\quad\text{as }p\rightarrow q.\label{eq_udecay}
\end{equation}
Using this and our previous estimates, we can now prove Theorem \ref{rescalednega}.


\begin{proof}[Proof of Theorem \ref{rescalednega}]
In light of the above results, and using the same notation as before, it suffices to show that $\tilde{u}_\infty=|Y|^{\frac{n-2}{4}}u_0$. Recall that $(\ol{M},\hat{g})$ has constant scalar curvature $\hat{R}=Y<0$. Since $R(\tilde{u}_\infty^{\frac{4}{n-2}} g_0)=-1$ by Proposition \ref{prop_utildeconverge}, we have that $v:=\frac{\tilde{u}_\infty}{|Y|^{\frac{n-2}{4}}u_0}$ satisfies
\begin{equation}
-a_n\hat{\Delta} v+Yv=Yv^{\frac{n+2}{n-2}}\label{eq_vPDE}
\end{equation}
on $\ol{M}\setminus\{q\}$. We will show that $v\equiv 1$. Recall from Lemma \ref{lower_bd_lem} that we already know $v\geq 1$. 

Fix a small neighborhood $U$ about $q$, and let $G_q$ be the Green's function of $\hat{\Delta}$ with pole at $q$ associated with the Dirichlet problem on $U$, as provided by Proposition \ref{prop_Green} (recall that $\hat{g}$ is $C^{1,\alpha}$ on $\ol{M}$ and smooth away from $q$). We will use the following properties of $G_q$:
\[-\hat{\Delta}G_q=\delta_q\quad\text{and}\quad G_q\geq 0\quad\text{ on }U,\quad G_q(p)\sim d(p,q)^{2-n}\quad\text{near }q.\]
Now we have for any $c\geq0$ and any $\epsilon>0$ that
\[-a_n\hat{\Delta}(v-1-c-\epsilon G_q)=-Yv+Yv^{\frac{n+2}{n-2}}\leq 0,\quad\text{on }U\setminus\{q\},\]
and since $v$ is smooth away from $q$ we can choose $c=\sup_{\partial U}v-1\geq 0$ so that
\[v-1-c-\epsilon G_q\leq v-1-c\leq 0\quad\text{on }\partial U.\]
Moreover, by our estimates on the growth of $v$ and $G_q$ near $q$, for any $\epsilon>0$ we can find a small neighbhorhood $V$ with $q\in V\subset U$ such that
\[v-1-c-\epsilon G_q\leq v-\epsilon G_q\leq 0\quad\text{on }\partial V.\]
Using the maximum principle, we have 
\[v-1-c\leq 0\quad\text{on } U\setminus V.\]
Now taking $\epsilon\rightarrow 0$, as $V$ can be chosen as a small ball tending to the point $q$, we conclude that
\[v-1-c\leq 0\quad\text{on } U\setminus\{q\}.\]
Therefore
\[\sup_{\ol{U}\setminus\{q\}}v-1=\sup_{\partial U}v-1.\]
As a result we deduce that $v-1$ achieves its maximum on $\ol{M}\setminus\{q\}$.

From this information we can now conclude that $v\equiv 1$. Indeed, if not then by the above we can take $p\in\ol{M}\setminus\{q\}$ where $v$ achieves its maximum, $v(p)>1$. But then \eqref{eq_vPDE} tells us that
\[(-a_n\hat{\Delta}v)(p)=-Yv(p)+Yv(p)^{\frac{n+2}{n-2}}<0,\]
which is impossible.

Finally, the uniqueness of $u_0$ is provided by Proposition \ref{prop_Yamabe_unique}. As a result, we have the convergence of $\tilde{u}$ as $t\rightarrow\infty$, and not just convergence up to subsequences.
\end{proof}

\section{The case $Y=0$}\label{sec_Y=0}

We now consider the remaining case in which $Y(M^n,[g_0])=0$, an assumption which will be implicit throughout this section, and prove Theorems \ref{rescaledzero} and \ref{rescaledzerolimit}. Let $u(x,t)$ be the solution to the Yamabe flow equation \eqref{flow_equation} as before. 

\subsection{Strictly slower blowup}\label{subsec_rescaledzero}

We have seen in the case $Y(M^n,[g_0])<0$ that $u(x,t)$ blows up at exactly the rate $t^{\frac{n-2}{4}}$. We will now see that in the case $Y(M^n,[g_0])=0$, the function $u(x,t)$ blows up less quickly, so that normalizing by $t^{-\frac{n-2}{4}}$ as we did now gives a limit which is uniformly zero, proving Theorem \ref{rescaledzero}.

Once again let $\tilde{u}_\infty$ be the pointwise limit of $\tilde{u}(x,t)=t^{-\frac{n-2}{4}}u(x,t)$ as discussed at the end of Section \ref{subsec_growth_control}. We first check that $\tilde{u}_\infty$ again satisfies an elliptic equation.

\begin{Lem}\label{lem_Y0_equation}
The function $\tilde{u}_\infty$ satisfies
\begin{equation}\label{eqn_Y0_equation}
-a_n\Delta_{g_0}\tilde{u}_\infty+R_{g_0}\tilde{u}_\infty=-\tilde{u}_\infty^{\frac{n+2}{n-2}}.
\end{equation}
in the weak sense.
\end{Lem}
\begin{proof}
Starting from the evolution equation \eqref{flow_equation} satisfied by $u$, by integrating against any $\varphi\in C_0^\infty(M)$ and dividing by $t^{\frac{n+2}{4}}$, we have
\begin{align*}
&\underbrace{t^{-\frac{n+2}{4}}\left(\int_M u(x,t)^{\frac{n+2}{n-2}}\varphi(x)\ dV_{g_0}-\int_M u(x,0)^{\frac{n+2}{n-2}}\varphi(x)\ dV_{g_0}\right)}_{A(t)}
\\
&=\underbrace{-\frac{n+2}{4}t^{-\frac{n+2}{4}}\int_0^t\int_Mu(x,s) L_{g_0}\varphi(x)\ dV_{g_0}\ ds}_{B(t)}.
\end{align*}
Clearly we have
\[\left|\int_M\tilde{u}_\infty(x)^{\frac{n+2}{n-2}}\varphi(x)\ dV_{g_0}-A(t)\right|\xrightarrow{t\rightarrow\infty}0.\]

We next claim that
\begin{equation}
\left|-\int_M\tilde{u}_\infty(x) L_{g_0}\varphi(x)\ dV_{g_0}-B(t)\right|\xrightarrow{t\rightarrow\infty}0.
\end{equation}
To see this, we compute
\begin{align*}
&-\frac{n+2}{4}t^{-\frac{n+2}{4}}\int_0^t\int_Mu(x,s) L_{g_0}\varphi(x)\ dV_{g_0}\ ds
\\
&=-\frac{n+2}{4}t^{-1}\int_0^t\int_M\tilde{u}(x,s)\frac{s^{\frac{n-2}{4}}}{t^{\frac{n-2}{4}}}L_{g_0}\varphi(x)\ dV_{g_0}\ ds
\\
&=-\frac{n+2}{4}\int_M\left(\int_0^1\tilde{u}(x,wt)w^{\frac{n-2}{4}}\ dw\right)L_{g_0}\varphi(x)\ dV_{g_0}
\\
&\xrightarrow{t\rightarrow\infty}-\frac{n+2}{4}\int_M\frac{4}{n+2}\tilde{u}_\infty(x)L_{g_0}\varphi(x)\ dV_{g_0}
\\
&\quad\qquad=-\int_M\tilde{u}_\infty(x)L_{g_0}\varphi(x)\ dV_{g_0}.
\end{align*}
where above in the third line we substituted $w=s/t$, and in the fourth line we used that $\tilde{u}(x,wt)w^{\frac{n-2}{4}}$ is integrable (for any $t>0$) and monotonically decreases to $\tilde{u}_\infty(x)w^{\frac{n-2}{4}}$ as $t\rightarrow\infty$ to apply the dominated convergence theorem.
\end{proof}

We can now prove Theorem \ref{rescaledzero}.

\begin{proof}[Proof of Theorem \ref{rescaledzero}]
Proceeding with an argument similar to that used in the proof of Theorem \ref{rescalednega}, we fix a small neighborhood $U$ about $q$ and let $G_q$ be the Green's function of $\hat{\Delta}$ with pole at $q$ associated with the Dirichlet problem on $U$. Then for any $c\geq 0$ we have that
\begin{equation}
    -a_n\hat{\Delta}(v-c-\epsilon G_q)=-v^{\frac{n+2}{n-2}}\leq 0,\quad\text{on }U\setminus\{q\},\label{eq_vPDE0}
\end{equation}
and we can choose $c=\sup_{\partial U} v\geq 0$ so that
\[v-c-\epsilon G_q\leq 0\quad\text{on }\partial U.\]
Moreover, by the estimates on the asymptotics of the growth of $v$ and $G_q$ near $q$, for any $\epsilon>0$ we can find a small neighborhood $V$ with $q\in V\subset U$ such that
\[v-c-\epsilon G_q\leq 0\quad\text{on }\partial V.\]
So again using the maximum principle, we have
\[v-c\leq 0\quad\text{on }U\setminus V.\]
Now taking $\epsilon\rightarrow 0$, as $V$ can be chosen as a small ball tending to the point $q$, we conclude that
\[v-c\leq 0\quad\text{on }U\setminus\{q\}.\]
This implies that $v$ achieves its maximum at some $p\in\ol{M}\setminus\{q\}$. But then by \eqref{eq_vPDE0} we have
\[0\leq (-a_n\hat{\Delta}v)(p)=-v(p)^{\frac{n+2}{n-2}}\leq 0.\]
Hence $\max_{\ol{M}} v= v(p)=0$ so that   $v=0$ and $\tilde{u}_\infty= 0$.
\end{proof}

\subsection{Convergence after a different rescaling}\label{subsec_rescaledzerolimit}

Above, we have seen that unlike in the $Y<0$ case, when $Y(M^n,[g_0])=0$ rescaling $u(x,t)$ by $t^{-\frac{n-2}{4}}$ only gives the trivial limit $\tilde{u}_\infty\equiv 0$. We now describe how to normalize $u(x,t)$ in a different way so as to identify a nontrivial limiting behavior as $t\rightarrow\infty$, proving Theorem \ref{rescaledzerolimit}. For this we need the following existence result for a certain conformal change of $g_0$ to another asymptotically flat metric with compactly supported non-positive scalar curvature as described below, which follows from a result of \cite{DiltsMaxwell}.

\begin{Lem}[{\cite[Theorem 5.1]{DiltsMaxwell}}]\label{lem_grho}
For any $K\subset M$ compact and when $\tau\in(0,n-2)$, there exists a metric $g_{\rho,0}=\rho^{\frac{4}{n-2}}g_0$ with $\rho-1\in C^{k+\alpha}_{-\tau}$ such that $R_{g_{\rho,0}}$ is compactly supported within $K$, and $\left.R_{g_{\rho,0}}\right|_K\leq 0$.
\end{Lem}
\begin{proof}
By \cite[Theorem 5.1]{DiltsMaxwell} and the embedding $C^{k+\alpha}_{-\tau}\subset W^{k,p}_{-\tau'}$ for any $\tau'\in(0,\tau)$, we have that the above result holds with $C^{k+\alpha}_{-\tau}$ replaced by $W^{k,p}_{-\tau'}$. Then by applying \cite[Lemma B.3]{ChenWang}, we see that moreover $\rho-1\in C^{k+\alpha}_{-\tau}$.
\end{proof}

We can then rewrite the evolution equation satisfied by $u$ in terms of the metric $g_{\rho,0}$ as
\begin{equation}
\frac{\partial}{\partial t} u_\rho^{\frac{n+2}{n-2}}=\frac{n+2}{4}(a_n\Delta_{g_{\rho,0}} u_\rho-R_{g_{\rho,0}}u_\rho),\label{mod_flow_eq}
\end{equation}
with $u_\rho(x,0)=\rho(x)^{-1}$. We will then have that
\begin{equation}
u_\rho(x,t)=\rho(x)^{-1}u(x,t).\label{eq_urho}
\end{equation}

Let $v=v(x,t)$ be the solution of \eqref{mod_flow_eq} satisfying the initial condition $v(x,0)\equiv 1$ with $v-1\in C^{k+\alpha}_{-\tau}$, and for $c>0$, set
\[v_c(x,t)=c v\left(x,c^{-\frac{4}{n-2}} t\right).\]
Then $v_c$ also solves \eqref{mod_flow_eq}.

We then have the following inequalities which follow directly from the comparison principle of Lemma \ref{max_principle_lem}.

\begin{Lem}\label{lem_uv_bound}
Let $u$ and $v$ be as above, and set
\begin{align*}
b&=\min_{x\in M}u_\rho(x,0)=\min_{x\in M}\rho^{-1}(x)\leq 1,
\\
B&=\max_{x\in M}u_\rho(x,0)=\max_{x\in M}\rho^{-1}(x)\geq 1.
\end{align*}
Then for all $t\geq 0$ and $x\in M$ we have
\[v_b(x,t)\leq u(x,t)\leq v_B(x,t).\]
\end{Lem}

In order to study the convergence of $v$ (and subsequently $u$) under appropriate rescalings, we first need control of the scalar curvatures of the associated metrics. Let
\[g_\rho(x,t)=v(x,t)^{\frac{4}{n-2}}g_{\rho,0}\]
be the family of metrics which make up the Yamabe flow starting from $g_{\rho,0}$.

\begin{Lem}\label{lem_R_bounds}
The scalar curvature of the metric $g_\rho$ satisfies
\[-\frac{1}{t}\leq R_{g_\rho}\leq 0.\]
\end{Lem}
\begin{proof}
Proposition \ref{lem_R_lower} gives us the lower bound. For the upper bound, we again apply the Ecker--Huisken maximum principle \cite[Theorem 4.3]{EckerHuisken}, this time directly to $\frac{\partial}{\partial t} R=(n-1)\Delta R+R^2$ on $M\times[0,T]$, where $T>0$ is arbitrary. Since $R_{g_\rho}(0)\leq 0$ by assumption, we conclude that $R_{g_\rho}\leq 0$ for all times.
\end{proof}

%
%
Although $v$ is not bounded in time, since it corresponds to a Yamabe flow starting from an asymptotically flat manifold with $Y(M^n,[g_0])=0$, we have that its maximum values remain in the compact set $K$, which contains the support of $R_{\rho_0}$.

\begin{Lem}\label{lem_max_compact}
Let $B(t)=\max_{x\in M}v(x,t)$, and let $K\subset M$ be a compact set. Then
\[B(t)=\max_{x\in K}v(x,t).\]
\end{Lem}
\begin{proof}
Since $R_{\rho,0}$ is supported within $K$ for all $t\geq 0$ and $R_{g_\rho}\leq 0$ by Lemma \ref{lem_R_bounds}, the function $v(x,t)-1$ is subharmonic on $M\setminus K$ and tends to zero at spatial infinity. So the maximum principle yields that
\[\sup_{x\in M\setminus K}v(x,t)-1\leq\max_{x\in\partial K}v(x,t)-1,\]
which leads to the desired conclusion.
\end{proof}

Next, with the help of the functions $v_b$ and $v_B$, we establish a Harnack inequality for $u_\rho$.

\begin{Prop}\label{prop_harnack}
For $p\in M$ and  $R>0$, there exists $C$ such that
\[\sup_{x\in B(p,R)} u_\rho(x,t)\leq C \inf_{x\in B(p,R)} u_\rho(x,t),\]
for any $t\geq 0$.
\end{Prop}
\begin{proof}
It suffices to prove a Harnack inequality for $v$. Indeed, a Harnack inequality for $v$ implies a Harnack inequality for both $v_b$ and $v_B$. Moreover, since
\[\frac{\partial}{\partial t} v=-\frac{n-2}{4} R_{g_\rho} v,\]
with $-\frac{1}{t}\leq R_{g_\rho}\leq 0$ by Lemma \ref{lem_R_bounds}, we have $v_B\leq\frac{B}{b}v_b$. Putting things together, we then see that (omitting the $x\in B(p,R)$ subscript) we would have
\[\sup u_\rho \leq \sup v_B\leq \frac{B}{b}\sup v_b\leq C\frac{B}{b}\inf v_b\leq C\frac{B}{b}\inf u_\rho,\]
as desired. So we will now prove that $v$ satisfies a Harnack inequality.

 Recall that $v$ satisfies
\begin{equation}
-a_n\Delta_{g_{\rho,0}}v+R_{g_{\rho,0}}v=R_{g_\rho} v^{\frac{4}{n-2}} v.\label{eq_v_elliptic}
\end{equation}
By Lemma \ref{lem_R_bounds} and Proposition \ref{upper_bd_lem}, we have that $\left|R_{g_\rho} v^{\frac{4}{n-2}}\right|$ is bounded by a uniform constant, while \eqref{eq_v_elliptic} is an elliptic equation, with the operator $\Delta_{g_{\rho,0}}$ defined using the fixed metric $g_{\rho,0}$. Therefore the Harnack inequality for $v$ follows (for instance, by \cite[Theorem 8.20]{GilbargTrudinger}).
\end{proof}
\begin{Cor}\label{cor_umaxbound}
We have that $\max_{x\in K}v(x,t)=o(t^{\frac{n-2}{4}})$, which in turn implies $\max_{x\in K}u(x,t)=o(t^{\frac{n-2}{4}})$. More generally, we have that $\max_{x\in M}v(x,t)=o(t^{\frac{n-2}{4}})$ and $\max_{x\in M}u(x,t)=o(t^{\frac{n-2}{4}})$.
\end{Cor}
\begin{proof}
By Proposition \ref{prop_harnack}, there exists some fixed $C>0$ such that for all times $t>0$ there holds
\begin{equation}
\sup_{x\in K} v(x,t)\leq C\inf_{x\in K}v(x,t).\label{eq_kharnack}
\end{equation}
Then for any $\epsilon>0$, by Theorem \ref{rescalednega} we can pick $x_0\in K$ and find $T>0$ large enough so that whenever $t\geq T$ we have
\[v(x_0,t)\leq\frac{\epsilon}{C}t^{\frac{n-2}{4}}.\]
So by \eqref{eq_kharnack} we conclude that when $t\geq T$ we also have $\sup_{x\in K}v(x,t)\leq\epsilon t^{\frac{n-2}{4}}$. Lemmas \ref{lem_max_compact} and \ref{lem_uv_bound} imply the decay of $\sup_{x\in M}v(x,t)$ and $\sup_{x\in M}u(x,t)$.
\end{proof}

Using the above facts, we can now show that we have a strictly positive subsequential limit of $\frac{u(x,t)}{\max_{x\in K}u(x,t)}$.

\begin{Prop}\label{prop_urho_holder}
The functions $\frac{u(x,t)}{\max_{x\in K}u(x,t)}$ subconverge in $C^{k,\alpha'}_{loc}$ for any $\alpha'<\alpha$ to a positive function $w(x)>0$ satisfying
\begin{equation}
-a_n\Delta_{g_{0}}w+R_{g_{0}} w\geq 0.\label{eq_w}
\end{equation}
\end{Prop}
\begin{proof}
For $\tau>0$, we define
\[L(\tau)=\max_{x\in K}u(x,\tau),\]
and
\[U_\tau(x,t)=L(\tau)^{-1}u\left((x,L(\tau)^{\frac{4}{n-2}}t+\tau\right).\]
Then $U_\tau$ also solves \eqref{flow_equation}. Moreover for $\tau>0$ sufficiently large, $U_\tau$ is uniformly bounded for $t\in\left[-\frac{1}{2},1\right]$ from above on $M$, and uniformly bounded from below away from zero on compact subsets of $M$. To see the upper bound, first observe that by Corollary \ref{cor_umaxbound},
\[L(\tau)^{\frac{4}{n-2}}\leq \tau\]
whenever $\tau$ is sufficiently large. Therefore
\[\sup_{t\in\left[-\frac{1}{2},1\right]}U_\tau(x,t)\leq L(\tau)^{-1}\sup_{s\in\left[\frac{\tau}{2},2\tau\right]}u(x,s).\]
Next, we have that
\[\sup_{s\in\left[\frac{\tau}{2},2\tau\right]}u(x,s)\leq B \sup_{s\in\left[\frac{\tau}{2},2\tau\right]} v_B(x,s)\leq B^2 2^{\frac{n-2}{4}} B(\tau),\]
using for the second inequality that $\frac{\partial}{\partial t} v=-\frac{n-2}{4} R_\rho v$.
Since we also have that
\[b^2 B(\tau)\leq b\sup_{x\in K}v_b(x,\tau)\leq L(\tau),\]
putting all these estimates together yields the upper bound
\begin{equation}
    U_\tau(x,t)\leq \frac{B^2}{b^2}2^{\frac{n-2}{4}},\label{eq_utau_bound}
\end{equation}
which hold for all $t\in \left[\frac{1}{2},1\right]$.

For the lower bound, we first note that positive constant multiples of $u(x,t)$ also satisfy a Harnack inequality as in Proposition \ref{prop_harnack} because $u_\rho(x,t)$ does, since
\begin{align*}
    \sup_{x\in B(p,R)}u(x,t)&\leq B\sup_{x\in B(p,R)}u_\rho(x,t)
    \\
    &\leq BC\inf_{x\in B(p,R)}u_\rho(x,t)\leq\frac{B}{b}C\inf_{x\in B(p,R)}u(x,t).
\end{align*}
We also have that
\[b^2\left(\frac{1}{2}\right)^{\frac{n-2}{4}}B(\tau)\leq b\sup v_b(x,\tau/2)\leq b\sup_{\substack{x\in K\\s\in\left[\frac{\tau}{2},2\tau\right]}}u_\rho(x,s)\leq\sup_{\substack{x\in K\\s\in\left[\frac{\tau}{2},2\tau\right]}}u(x,s),\]
and that
\[L(\tau)\leq B\sup_{x\in K}u_\rho(x,\tau)\leq B^2 B(\tau).\]
Hence
\[\frac{b^2}{B^2}\left(\frac{1}{2}\right)^{\frac{n-2}{4}}\leq L(\tau)^{-1}\sup_{\substack{x\in K\\s\in\left[\frac{\tau}{2},2\tau\right]}}u(x,s).\]
Applying the Harnack inequality and translating back, this shows that $U_\tau$ is indeed bounded from below away from zero uniformly on compact sets for $t\in\left[-\frac{1}{2},1\right]$.

Because $U_\tau$ is uniformly bounded above and below away from zero on compact sets, we can now apply local parabolic Krylov-Safonov and Schauder estimates to \eqref{flow_equation} to obtain uniform local H\"{o}lder estimates on $\ol{\Omega}\times\left[\frac{1}{2},1\right]$. In particular we have uniform $C^{k,\alpha}_{loc}$ control of $U_\tau(x,0)=\frac{u(x,\tau)}{\max_{x\in K}u(x,\tau)}$ and therefore subconvergence in $C^{k,\alpha'}_{loc}$ for any $\alpha'<\alpha$ to some $w(x)>0$ as $\tau\rightarrow\infty$. The positivity of $w(x)$ is a consequence of our uniform lower bound on $\frac{u(x,\tau)}{\max_{x\in K}u(x,\tau)}$ on compact sets.

To verify the equation satisfied by $w$, recall that 
\begin{align*}
&-a_n\Delta_{g_{0}}\frac{u(x,t)}{\max_{x\in K}u(x,t)}+R_{g_{0}}\frac{u(x,t)}{\max_{x\in K}u(x,t)}
\\
&\quad=R_{g(t)}u(x,t)^{\frac{4}{n-2}}\frac{u(x,t)}{\max_{x\in K}u(x,t)}.
\end{align*}
On the right-hand side we have
$R_{g(t)}\geq-\frac{1}{t}$ by Lemma \ref{lem_R_lower} and $u_\rho(x,t)=o(t^{\frac{n-2}{4}})$ by Theorem \ref{rescaledzero}. So it follows that in the limit we have \eqref{eq_w} in the classical sense.


\end{proof}

From the discussion leading up to \eqref{eq_u0}, we know there exists a function $u_0$ which makes \eqref{eq_w} an equality. The following uniqueness result will allow us to identify the limit of the rescaling $\frac{u_{\rho}(x,t)}{\max_{x\in K}u(x,t)}$.


\begin{Prop}\label{Prop_uniqueness}
Suppose $(M,g)$ is a $W^{2,p}_{-\tau}$ AF manifold for some $p>n/2$ and $\tau\geq\frac{n}{2}-2$, with $Y(M,[g])=0$. Recall that there is a positive function $u_0$ solving the equation
\[-a_n\Delta_g u_0+R_gu_0=0,\] such that 
$u_0=O(|z|^{2-n})$ as $|z|\rightarrow\infty$.
Then if $w$ is a positive function on $(M,g)$ in $W^{1,2}_{loc}$ such that 
\[-a_n\Delta_g w+R_gw\ge 0,\]
then there is a positive constant $c$ such that 
$$w=cu_0.$$
\end{Prop}
\proof
We write $w=\psi u_0$ and  a classical calculation yields the equality:
\begin{align*}
-a_n\Delta w+R_gw&=\left(-a_n\Delta u_0+R_gu_0\right)\psi-2a_n\langle du_0,d\psi\rangle-a_nu_0\Delta \psi\\
&=-2a_n\langle du_0,d\psi\rangle-a_nu_0\Delta \psi\\
&=-a_nu_0^{-1}\text{div}\left( u_0^2d\psi\right).
\end{align*}
Integration this identity against $\varphi^2 \psi^{-1} u_0$ where $\varphi$ is a non negative compactly supported function and taking in account our hypothesis, one gets
$$\int_{M} \varphi^2 \frac{|d\psi|^2_g}{\psi^2} u_0^2 dV_{g}\le 2\int_{M} \varphi\langle d\varphi,d\psi\rangle\frac{ u_0^2}{\psi} dV_{g}.$$
And with the Cauchy-Schwarz inequality, one obtains the inequality
$$\int_{M} \varphi^2 \frac{|d\psi|^2_g}{\psi^2} u_0^2 dV_{g}\le 4 \int_{M}|d\varphi|^2_g u_0^2 dV_{g}.$$
Using a sequence of functions $\varphi_\ell$ satisfying
$$\begin{cases} \varphi_\ell=1& \text{ on }~|z|\le \ell\\
 \varphi_\ell=0& \text{ on }~|z|\ge 2\ell\\
|d\varphi_\ell|\le 2/\ell & \text{on } \ell \le |z|\le 2\ell .\end{cases}$$
together with  the decay estimate  of $u_0$ and the fact that
$$V_{g}\left(\{\ell \le |z|\le 2\ell\}\right)=O\left(\ell^{n}\right)$$ one obtains that 
$$\int_{\{ |z|\le \ell\} }  \frac{|d\psi|^2_g}{\psi^2} u_0^2 dV_{g}\le O(\ell^{2-n})$$
Hence letting $\ell\to+\infty$, one deduces that 
$d\psi=0$.
\endproof

Putting everything together, we can now prove Theorem \ref{rescaledzerolimit}.

\begin{proof}[Proof of Theorem \ref{rescaledzerolimit}]
This follows immediately from Propositions \ref{prop_urho_holder} and \ref{Prop_uniqueness}; note that $\sup_{x\in K}\frac{u(x,t)}{\max_{x\in K}u(x,t)}=1$ for all $t$ implies $\sup_{x\in K}w(x)=1$. The uniqueness of $u_0$ up to scaling is provided by either Proposition \ref{Prop_uniqueness} or Proposition \ref{prop_Yamabe_unique}. Because any convergent subsequence of $\frac{u(x,t)}{\max_{x\in K}u(x,t)}$ must converge to the same limit, we in fact have convergence of $\frac{u(x,t)}{\max_{x\in K}u(x,t)}$ as $t\rightarrow\infty$ to $w(x)$.

\end{proof}

\bibliographystyle{alpha}
\bibliography{references}

\begin{thebibliography}{ACM14}

\bibitem[ACM14]{ACM}
Kazuo Akutagawa, Gilles Carron, and Rafe Mazzeo.
\newblock The {Y}amabe problem on stratified spaces.
\newblock {\em Geom. Funct. Anal.}, 24(4):1039--1079, 2014.

\bibitem[Bar86]{Bartnik86}
Robert Bartnik.
\newblock The mass of an asymptotically flat manifold.
\newblock {\em Comm. Pure Appl. Math.}, 39(5):661--693, 1986.

\bibitem[Bre05]{Brendle5}
Simon Brendle.
\newblock Convergence of the {Y}amabe flow for arbitrary initial energy.
\newblock {\em J. Differential Geom.}, 69(2):217--278, 2005.

\bibitem[Bre07]{Brendle6}
Simon Brendle.
\newblock Convergence of the {Y}amabe flow in dimension 6 and higher.
\newblock {\em Invent. Math.}, 170(3):541--576, 2007.

\bibitem[BV19]{Bahuaud}
Eric Bahuaud and Boris Vertman.
\newblock Long-time existence of the edge {Y}amabe flow.
\newblock {\em J. Math. Soc. Japan}, 71(2):651--688, 2019.

\bibitem[CB81]{CantorBrill}
Murray Cantor and Dieter Brill.
\newblock The {L}aplacian on asymptotically flat manifolds and the
  specification of scalar curvature.
\newblock {\em Compositio Math.}, 43(3):317--330, 1981.

\bibitem[Che19]{EC}
Eric Chen.
\newblock Convergence of the ricci flow on asymptotically flat manifolds with
  integral curvature pinching, 2019.

\bibitem[Cho92]{Chow2}
Bennett Chow.
\newblock Yamabe flow on locally conformally flat manifolds with positive ricci
  curvatur.
\newblock {\em Comm. Pure. Appl. Math}, 45:1003--1014, 1992.

\bibitem[COV21]{CLV}
Gilles {Carron}, J{\o}rgen {Olsen Lye}, and Boris {Vertman}.
\newblock {Convergence of the Yamabe flow on singular spaces with positive
  Yamabe constant}.
\newblock {\em arXiv e-prints}, page arXiv:2106.01799, June 2021.

\bibitem[CW21]{ChenWang}
Eric {Chen} and Yi~{Wang}.
\newblock {The Yamabe flow on asymptotically flat manifolds}.
\newblock {\em arXiv e-prints}, page arXiv:2102.07717, February 2021.

\bibitem[CZ15]{CZ}
Liang Cheng and Anqiang Zhu.
\newblock Yamabe flow and {ADM} mass on asymptotically flat manifolds.
\newblock {\em J. Math. Phys.}, 56(10):101507, 21, 2015.

\bibitem[DM07]{DaiMa}
Xianzhe Dai and Li~Ma.
\newblock Mass under the {R}icci flow.
\newblock {\em Comm. Math. Phys.}, 274(1):65--80, 2007.

\bibitem[DM18]{DiltsMaxwell}
James Dilts and David Maxwell.
\newblock Yamabe classification and prescribed scalar curvature in the
  asymptotically {E}uclidean setting.
\newblock {\em Comm. Anal. Geom.}, 26(5):1127--1168, 2018.

\bibitem[EH91]{EckerHuisken}
Klaus Ecker and Gerhard Huisken.
\newblock Interior estimates for hypersurfaces moving by mean curvature.
\newblock {\em Invent. Math.}, 105(3):547--569, 1991.

\bibitem[GT83]{GilbargTrudinger}
David Gilbarg and Neil~S. Trudinger.
\newblock {\em Elliptic partial differential equations of second order}, volume
  224 of {\em Grundlehren der mathematischen Wissenschaften [Fundamental
  Principles of Mathematical Sciences]}.
\newblock Springer-Verlag, Berlin, second edition, 1983.

\bibitem[Ham89]{Hamilton}
Richard Hamilton.
\newblock Lecture notes on heat equations in geometry.
\newblock {\em unpublished manuscript}, 1989.

\bibitem[Her97]{Herzlich}
Marc Herzlich.
\newblock Compactification conforme des vari\'{e}t\'{e}s asymptotiquement
  plates.
\newblock {\em Bull. Soc. Math. France}, 125(1):55--91, 1997.

\bibitem[Li18]{YuLi}
Yu~Li.
\newblock Ricci flow on asymptotically {E}uclidean manifolds.
\newblock {\em Geom. Topol.}, 22(3):1837--1891, 2018.

\bibitem[LV20]{Olsen}
Jørgen~Olsen Lye and Boris Vertman.
\newblock Long-time existence of yamabe flow on singular spaces with positive
  yamabe constant, 2020.

\bibitem[Ma19]{MaLi}
Li~Ma.
\newblock Yamabe flow and metrics of constant scalar curvature on a complete
  manifold.
\newblock {\em Calc. Var. Partial Differential Equations}, 58(1):Paper No. 30,
  16, 2019.

\bibitem[Ma21]{MaLi21}
Li~Ma.
\newblock Global {Y}amabe flow on asymptotically flat manifolds.
\newblock {\em J. Funct. Anal.}, 281(10):Paper No. 109229, 14, 2021.

\bibitem[Max05]{Maxwell}
David Maxwell.
\newblock Solutions of the {E}instein constraint equations with apparent
  horizon boundaries.
\newblock {\em Comm. Math. Phys.}, 253(3):561--583, 2005.

\bibitem[PW84]{Protter}
Murray~H. Protter and Hans~F. Weinberger.
\newblock {\em Maximum principles in differential equations}.
\newblock Springer-Verlag, New York, 1984.
\newblock Corrected reprint of the 1967 original.

\bibitem[Sau12]{Sauvigny}
Friedrich Sauvigny.
\newblock {\em Partial differential equations. 2}.
\newblock Universitext. Springer-Verlag London, Ltd., London, 2012.
\newblock Functional analytic methods, With consideration of lectures by E.
  Heinz, Second revised and enlarged edition of the 2006 translation.

\bibitem[Sch19]{Schulz}
Mario~B. Schulz.
\newblock Instantaneously complete {Y}amabe flow on hyperbolic space.
\newblock {\em Calc. Var. Partial Differential Equations}, 58(6):Paper No. 190,
  30, 2019.

\bibitem[Sch20]{Schulz2}
Mario~B. Schulz.
\newblock Unconditional existence of conformally hyperbolic {Y}amabe flows.
\newblock {\em Anal. PDE}, 13(5):1579--1590, 2020.

\bibitem[SS03]{Struwe}
Hartmut Schwetlick and Michael Struwe.
\newblock Convergence of the {Y}amabe flow for ``large'' energies.
\newblock {\em J. Reine Angew. Math.}, 562:59--100, 2003.

\bibitem[Stu94]{Strum}
Karl-Theodor Sturm.
\newblock Analysis on local {Dirichlet} spaces. {I}: {Recurrence},
  conservativeness and {{\(L^ p\)}}-{Liouville} properties.
\newblock {\em J. Reine Angew. Math.}, 456:173--196, 1994.

\bibitem[Ye94]{Ye}
Rugang Ye.
\newblock Global existence and convergence of {Y}amabe flow.
\newblock {\em J. Differential Geom.}, 39(1):35--50, 1994.

\end{thebibliography}

\end{document}